\documentclass[a4paper,12pt]{amsart}
\usepackage{amssymb,amsmath,amsfonts, amsthm}
\usepackage{hyperref}
\usepackage{dsfont}
\usepackage{upgreek}
\usepackage{mathrsfs}
\usepackage{mathabx,yfonts}

\newtheorem{thm}{Theorem}

\newtheorem{lem}{Lemma}

\newtheorem*{ack}{Acknowledgements}

\numberwithin{equation}{section}

\begin{document}

\title{On the saturation number for singular cubic surfaces}

\author{Yuchao Wang}

\address{
Department of Mathematics, Shanghai University, Shanghai, 200444, China}
\email{yuchaowang@shu.edu.cn}

\author{Weili Yao}

\address{
Department of Mathematics, Shanghai University, Shanghai, 200444, China}
\email{yaoweili@shu.edu.cn}

\begin{abstract}
We investigate the distribution of rational points on singular cubic surfaces, whose coordinates have few prime factors. The key tools used are universal torsors, the circle method and results on linear equations in primes.
\end{abstract}



\maketitle

\section{Introduction}
The study of prime solutions of Diophantine equations is a topic of central importance in number theory. One example is Vinogradov's celebrated proof of the three primes theorem \cite{Vino} via the circle method. Subsequently, Hua \cite{Hua} extended this approach to investigate the problem of representing integers by powers of primes. For more applications to nonlinear diagonal equations, see the work of Kawada and Wooley \cite{KW} and Zhao \cite{Zhao} for further developments and references in the Waring-Goldbach problem. Moreover, Liu \cite{Liu} and Zhao \cite{Zhao2} considered prime solutions of one indefinite quadratic equation, which is not necessarily diagonal. For more general cases, Cook and Magyar \cite{CM} established the asymptotic estimate for the number of prime solutions of general systems of Diophantine equations, under the assumption that the number of variables is considerably large compared with the degrees of the polynomials. For developments in this topic, one may refer to the recent work of Yamagishi \cite{Y1}, Liu and Zhao \cite{LiuZhao} and Green \cite{Green}.

In the analysis of prime solutions of Diophantine equations via the circle method, a large number of variables are required, due to the barrier for cancellation in the estimate of exponential sums. Thus when there are fewer variables, we may focus on integral solutions with few prime factors. One significant contribution in this realm is the result by Chen \cite{Chen} on the binary Goldbach conjecture, obtained by using sieve methods. More generally, Bourgain, Gamburd and Sarnak \cite{BGS} raised the question of whether rational points or integral points with almost prime coordinates form a Zariski dense subset in suitable varieties. Bourgain, Gamburd and Sarnak \cite{BGS} and Nevo and Sarnak \cite{NS} considered such points on orbits of congruence subgroups of semi-simple groups acting linearly on affine space. Furthermore, Liu and Sarnak \cite{LS} and Hong and Kontorovich \cite{HK} investigated integral points whose coordinates have few prime factors on certain affine quadrics. However, these results do not cover the case of cubic surfaces considered here.

Given an irreducible cubic surface $X\subset \mathbb{P}^3$, one may find an absolutely irreducible cubic form $C(x_0,x_1,x_2,x_3)\in \mathbb{Z}[x_0,x_1,x_2,x_3]$ such that $X$ is defined by the equation $C(x_0,x_1,x_2,x_3)=0$. A cubic surface $X$ is singular if and only if
\begin{equation*}
\frac{\partial C}{\partial x_0}=\frac{\partial C}{\partial x_1}=\frac{\partial C}{\partial x_2}=\frac{\partial C}{\partial x_3}=0
\end{equation*}
has nonzero solutions.

In this paper, we are concerned with the distribution of almost prime points on certain singular cubic surfaces, whose universal torsors are open subsets of hypersurfaces in affine space. The classification of singular cubic surfaces was done by Schl\"{a}fli \cite{Sch} and Cayley \cite{Cayley}, and one may refer to Bruce and Wall \cite{BW} for a modern treatment. There are a finite number of types, for which we may describe each type by the types of the $\mathbf{ADE}$-singularities and the number of lines lying in the corresponding singular cubic surfaces. It follows from Theorem 2 of Derenthal \cite{Der} that there are seven types of singular cubic surfaces whose Cox rings have a minimal set of $10$ generators with one relation. Note that there are two isomorphy classes of cubic surfaces with $\mathbf{D}_4$ singularity type, and we let $X_1$ and $X_2$ denote these two isomorphy classes. For $3\leq i\leq8$, we write $X_i$ for cubic surfaces with singularities of types $\mathbf{A}_3+2\mathbf{A}_1$, $2\mathbf{A}_2+\mathbf{A}_1$, $\mathbf{A}_4+\mathbf{A}_1$, $\mathbf{D}_5$, $\mathbf{A}_5+\mathbf{A}_1$ and $\mathbf{E}_6$, respectively. We remark that the equation $C(x_0,x_1,x_2,x_3)=0$ defining each cubic surface is not unique, and we will use the defining equations listed in Section 3.5 of Derenthal \cite{Der}.

Let $P_r$ denote an $r$-almost prime, which is a nonzero integer with at most $r$ prime factors, counted with multiplicity. Moreover, we let $\mathbb{Z}^4_{\text{prim}}$ be the set of vectors $\mathbf{x}=(x_0,x_1,x_2,x_3)\in\mathbb{Z}^4$ satisfying $\text{gcd}(x_0,x_1,x_2,x_3)=1$.  For any given cubic surface $X$, we define the saturation number $r(X)$ to be the least natural number $r$ such that the set of $\mathbf{x}\in \mathbb{Z}^4_{\text{prim}}$, for which $[\mathbf{x}]\in X$ and $x_0x_1x_2x_3=P_r$, is Zariski dense in the cubic surface $X$. This definition is essentially due to Bourgain, Gamburd, and Sarnak \cite{BGS}.

Since we are seeking for the set of almost prime points which form a Zariski dense subset in the cubic surface, it is not enough to only work with such points lying on finitely many curves contained in the surface. On writing $U_i$ for the complement of the lines in each surface $X_i$, we may restrict our attention to the open subset $U_i\subset X_i$, for $1\leq i\leq 8$. In order to prove that a set of $[\mathbf{x}]\in X_i$ is Zariski dense, it suffices to show that for any given $\varepsilon>0$ and any $\boldsymbol{\xi}_i\in\mathbb{R}^4$ with $[\boldsymbol{\xi}_i]\in U_i$, there exists $B\in \mathbb{N}$ sufficiently large and at least one point $[\mathbf{x}]$ in the set, such that
\begin{equation*}
\Big|\frac{\mathbf{x}}{B}-\boldsymbol{\xi}_i\Big|< \varepsilon.
\end{equation*}
Here we use $|\mathbf{y}|$ to denote the maximum of the absolute values of each component of the vector $\mathbf{y}$. 

The aim of this paper is to establish the upper bounds for the saturation numbers $r(X_i)$, for $1\leq i\leq 8$. The following is the main result.
\begin{thm}
\label{theorem}
Let $1\leq i\leq 8$ and $[\boldsymbol{\xi}_i]$ be a real point on $U_i$. Set $r_1=r_6=r_7=12$, $r_2=r_4=r_5=13$, $r_3=14$ and $r_8=29$. For any $\varepsilon>0$, we define
\begin{equation*}
M_{U_i}(\boldsymbol{\xi}_i,\varepsilon,B,r)=\# \left\{\mathbf{x}\in\mathbb{Z}^4_{\text{\em{prim}}}:
\begin{aligned}
&[\mathbf{x}]\in U_i,\,\Big|\frac{\mathbf{x}}{B}-\boldsymbol{\xi}_i\Big|<\varepsilon,\\
&x_0x_1x_2x_3=P_{r}
\end{aligned}
\right\}.
\end{equation*}
Then for sufficiently large $B$, we have
\begin{equation*}
M_{U_i}(\boldsymbol{\xi}_i,\varepsilon,B,r_i)\gg B(\log B)^{-4},
\end{equation*}
for $i\neq 2,8$ and
\begin{equation}
\label{eq2}
M_{U_2}(\boldsymbol{\xi}_2,\varepsilon,B,r_2)\gg B(\log B)^{-5},
\end{equation}
\begin{equation*}
M_{U_8}(\boldsymbol{\xi}_8,\varepsilon,B,r_8)\gg B^{\frac{3}{7}}(\log B)^{-4}.
\end{equation*}
The implicit constants may depend on $\boldsymbol{\xi}_i$ and $\varepsilon$. Therefore, we establish the upper bounds for the saturation numbers $r(X_i)$, as shown in Table 1.
\begin{table}[!hbp]
\begin{tabular}{c c c }
\hline
Surfaces&Singularities& Upper bounds for $r(X_i)$   \\
\hline
$X_1$&$\mathbf{D}_4$, first isomorphy class & $12$  \\
$X_2$&$\mathbf{D}_4$, second isomorphy class & $13$  \\
$X_3$&$\mathbf{A}_3+2\mathbf{A}_1$ & $14$  \\
$X_4$&$2\mathbf{A}_2+\mathbf{A}_1$ & $13$ \\
$X_5$&$\mathbf{A}_4+\mathbf{A}_1$ & $13$ \\
$X_6$&$\mathbf{D}_5$& $12$\\
$X_7$&$\mathbf{A}_5+\mathbf{A}_1$&$12$\\
$X_8$&$\mathbf{E}_6$&$29$\\
\hline
\end{tabular}
\caption{Upper bounds for the saturation numbers}
\end{table}
\end{thm}
Results of this type were first considered by Wang \cite{Wang}, in which the upper bounds for the saturation numbers for the Fermat cubic surface and the Cayley cubic surface are obtained. Note that the Fermat cubic surface is smooth, while the Cayley cubic surface has singularity type $4\mathbf{A}_1$. Subsequently, Sofos and Wang \cite{SW} proved that any smooth projective variety has finite saturation number, provided that it is $\mathbb{Q}$-unirational. However, the explicit upper bound for the saturation number is rather large for general cases.

It is of interest to compare our result with the density of rational points on cubic surfaces. Set
\begin{equation*}
N_{U}(B)=\#\{\mathbf{x}\in \mathbb{Z}^4_{\text{prim}}:\,[\mathbf{x}]\in U,\,\max |x_i|\leq B \},
\end{equation*}
which counts the number of rational points of bounded height on cubic surfaces. Manin's conjecture (see Batyrev and Manin \cite{BM}) predicts that there exists some suitable positive constant $c_X$ such that
\begin{equation*}
N_{U}(B)\sim c_XB(\log B)^{\rho_X-1},
\end{equation*}
as $B\rightarrow \infty$, where $\rho_X$ is the rank of the Picard group of the cubic surface $X$. For the cubic surface with $\mathbf{D}_4$ singularity type, Browning \cite{Bro} achieved the upper and lower bound for the number of rational points of bounded height, for which the order of magnitude agrees with the prediction. Later, Le Boudec \cite{LeB} proved Manin's conjecture for the cubic surface of this type. Moreover, for cubic surfaces with singularities of types $2\mathbf{A}_2+\mathbf{A}_1$, $\mathbf{D}_5$, $\mathbf{A}_5+\mathbf{A}_1$ and $\mathbf{E}_6$, Manin's conjecture is established by Le Boudec \cite{LeB2},  Browning and Derenthal \cite{BD}, Baier and Derenthal \cite{BaiD} and de la Bret\`{e}che, Browning and Derenthal \cite{BBD}, respectively. One may refer to Frei \cite{Frei} for additional references. However, Manin's conjecture is still open for several singular cubic surfaces, including cubic surfaces with singularities of types $\mathbf{A}_3+2\mathbf{A}_1$ and $\mathbf{A}_4+\mathbf{A}_1$.

The proof of our main result relies on convenient parametrisations of rational points on singular cubic surfaces $X_1,\dots,X_8$. We first apply the theory of universal torsors to specify certain integral solutions in particular forms, whose coordinates have few prime factors. Then we establish a lower bound for the number of such solutions which are close to some fixed real solution via the circle method and results on linear equations in primes due to Green and Tao \cite{GT}. More precisely, we use the circle method to count the number of solutions to the equation
\begin{equation}
\label{F1}
2\beta_1 p_1+\beta_2p_2^2+\beta_3p_3p_4=0,
\end{equation}
and apply results on linear equations in primes to investigate the solutions to the equations
\begin{equation}
\label{F2}
2\beta_1 p_1+\beta_2p_2+\beta_3p_3=0,
\end{equation}
\begin{equation}
\label{F3}
\beta_1 p_1+\beta_2p_2+\beta_3p_3+\beta_4p_4=0
\end{equation}
and
\begin{equation}
\label{F4}
    \left\{\begin{array}{ll}
&\beta_1p_1+2\beta_4p_4-\beta_5p_5=0, \\
& \beta_2p_2+\beta_3p_3+2\beta_4p_4=0,
    \end{array}\right. \\
\end{equation}
where $p_j$ are primes lying in specified intervals and $\beta_j\in\{1,-1\}$, for $1\leq j\leq5$. We remark that we may also use the circle method to investigate the equations (\ref{F2})--(\ref{F4}). However, it will be more straightforward to apply the results due to Green and Tao \cite{GT} instead. The result for $r(X_8)$ seems worse than others, since we do not get parametrisations of rational points, which yields a better upper bound for $r(X_8)$ and can be dealt with by using the circle method.

For singular cubic surfaces $X_1,\dots,X_{8}$, we define $\tilde{r}(X_i)$ to be the least natural number $\tilde{r}$ such that the set of $\mathbf{x}\in \mathbb{Z}^4_{\text{prim}}$ for which $[\mathbf{x}]\in X_i$ and the product $x_0x_1x_2x_3$ has at most $\tilde{r}$ \emph{distinct} odd prime factors, is Zariski dense in $X_i$. Then it is worth pointing out that our methods give $\tilde{r}(X_2)\leq 5$ and $\tilde{r}(X_i)\leq 4$ for $i\neq2$. This corresponds to the exponents of $\log B$ in the lower bounds for $M_{U_i}(\boldsymbol{\xi}_i,\varepsilon,B,r_i)$ in Theorem \ref{theorem}.

Browning and Gorodnik \cite{BG} defined the permeation number to investigate power-free values of polynomials on symmetric varieties. In a similar manner, for any cubic surface $X$, we may define the permeation number $r^{\Box}(X)$ to be the least natural number $r'$ such that the set of $\mathbf{x}\in \mathbb{Z}^4_{\text{prim}}$, for which $[\mathbf{x}]\in X$ and $x_0x_1x_2x_3$ is $r'$-free, is Zariski dense in the cubic surface $X$. We remark that our methods may be applied to obtain the upper bound for $r^{\Box}(X_i)$ with $1\leq i\leq 8$. 

An outline of this paper is as follows. We devote Section 2 to obtaining parametrisations of rational points on the singular cubic surfaces by using the theory of universal torsors. In Section 3, we employ the circle method to solve the equation (\ref{F1}) in prime variables which lie in certain intervals. In Section 4, we apply results on linear equations in primes to investigate the solutions to the equations (\ref{F2})--(\ref{F4}). The proof of Theorem \ref{theorem} is completed in Section 5.

Throughout the paper, we let the letter $p$, with or without indices, be reserved for prime numbers. Let $\varepsilon$ represent a small positive constant, not necessarily the same in all occurrences. As usual, let $\mu(n)$, $\varphi(n)$ and $d(n)$ denote the M\"{o}bius function, Euler's totient function and the divisor function, respectively. For any $j$, we let $\beta_j\in\{-1,1\}$. We also write $e(\alpha)=e^{2\pi i \alpha}$ and $(a,b)=\text{gcd}(a,b)$.

\begin{ack}
\emph{The authors are grateful to the referees for carefully reviewing the paper and various suggestions.}
\end{ack}

\section{The universal torsor}
Universal torsors were originally introduced by Colliot-Th\'{e}l\`{e}ne and Sansuc to investigate the Hasse principle and weak approximation for rational varieties. The parametrisations of rational points obtained by universal torsors have been widely used in the context of Manin's conjecture to study the asymptotic behavior of the number of rational points of bounded height on Fano varieties.

In this section, we use a passage to the universal torsor for the singular cubic surfaces $X_1,\dots,X_8$. Details can be found in Section 3.5 of Derenthal \cite{Der} and we employ the same notations used in \cite{Der}. We remark that we are not working with the full universal torsor and only focus on convenient parametrisations of rational points, which form a Zariski dense subset in the cubic surface.

Type $\mathbf{D}_4$, first isomorphy class. The surface $X_1$ is defined by the equation
\begin{equation}
\label{D4}
x_0(x_1+x_2+x_3)^2-x_1x_2x_3=0.
\end{equation}
Browning \cite{Bro} and Le Boudec \cite{LeB} considered the distribution of rational points on the cubic surface $X_1$ by using the theory of universal torsors and analytic tools. It is shown in Section 3 of Le Boudec \cite{LeB} that the universal torsor for the cubic surface $X_1$ is an open subvariety in $\mathbb{A}^{10}=\text{Spec}\mathbb{Z}[\eta_1,\dots,\eta_{10}]$, defined by the equation
\begin{equation}
\label{torsor}
\eta_2\eta_5^2\eta_8+\eta_3\eta_6^2\eta_9+\eta_4\eta_7^2\eta_{10}-\eta_1\eta_2\eta_3\eta_4\eta_5\eta_6\eta_7=0,
\end{equation}
where $\eta_1,\dots,\eta_7>0$, $\eta_8\eta_9\eta_{10}\neq0$ and $\eta_1,\dots,\eta_{10}$ satisfy necessary coprimality conditions (3.2)--(3.7) in \cite{LeB}. Then the vector given by
\begin{equation*}
(\eta_8\eta_9\eta_{10},\eta_1^2\eta_2^2\eta_3\eta_4\eta_5^2\eta_8,\eta_1^2\eta_2\eta_3^2\eta_4\eta_6^2\eta_9,\eta_1^2\eta_2\eta_3\eta_4^2\eta_7^2\eta_{10})
\end{equation*}
is a primitive integral solution of the equation (\ref{D4}). Now we fix $\eta_2=\dots=\eta_7=1$, then the relation (\ref{torsor}) becomes
\begin{equation}
\label{torsor2}
\eta_8+\eta_9+\eta_{10}-\eta_1=0
\end{equation}
with $\eta_1>0$, $\eta_8\eta_9\eta_{10}\neq0$ and $(\eta_1,\eta_8)=(\eta_1,\eta_9)=(\eta_1,\eta_{10})=1$. Moreover, we set
\begin{equation*}
\mathbf{x}=(x_0,x_1,x_2,x_3)=(\eta_8\eta_9\eta_{10},\eta_1^2\eta_8,\eta_1^2\eta_9,\eta_1^2\eta_{10})
\end{equation*}
with $\eta_1,\eta_8,\eta_9$ and $\eta_{10}$ satisfying (\ref{torsor2}). Then we have $[\mathbf{x}]\in X_1$.

Arguing similarly, we may get parametrisations of certain rational points on each cubic surface $X_i$, for $2\leq i\leq 8$. We omit the details and only give the sketch.

Type $\mathbf{D_4}$, second isomorphy class. The defining equation for the surface $X_2$ is given by
\begin{equation*}
x_0(x_1+x_2+x_3)^2+x_1x_2(x_1+x_2)=0.
\end{equation*}
We fix
\begin{equation}
\label{para2}
\mathbf{x}=(x_0,x_1,x_2,x_3)=(\eta_8\eta_9\eta_{10},\eta_1^2\eta_8,\eta_1^2\eta_9,\eta_1^2(\eta_1+\eta_{10}))
\end{equation}
with $\eta_8+\eta_9+\eta_{10}=0$. We see that $[\mathbf{x}]$ lies on $X_2$.

Type $\mathbf{A_3+2A_1}$. The surface $X_3$ is defined by the equation
\begin{equation*}
x_3^2(x_1+x_2)+x_0x_1x_2=0.
\end{equation*}
Write
\begin{equation*}
\mathbf{x}=(x_0,x_1,x_2,x_3)=(\eta_1\eta_{10}^2,\eta_4^2\eta_8,\eta_1\eta_4\eta_8,\eta_1\eta_4\eta_{10})
\end{equation*}
with $\eta_1+\eta_4+\eta_{8}=0$.  Thus we obtain $[\mathbf{x}]\in X_3$.

Type $\mathbf{2A_2+A_1}$. The equation defining the surface $X_4$ is given by
\begin{equation*}
x_3^2(x_1+x_3)+x_0x_1x_2=0.
\end{equation*}
On writing
\begin{equation*}
\mathbf{x}=(x_0,x_1,x_2,x_3)=(\eta_5\eta_{6}^2,\eta_2\eta_5^2,\eta_2^2\eta_{10},\eta_2\eta_5\eta_{6}),
\end{equation*}
where $\eta_5+\eta_6+\eta_{10}=0$, we get $[\mathbf{x}]\in X_4$.

Type $\mathbf{A_4+A_1}$. For the surface $X_5$, the defining equation is
\begin{equation*}
x_2x_3^2+x_1^2x_3+x_0x_1x_2=0.
\end{equation*}
Fix
\begin{equation*}
\mathbf{x}=(x_0,x_1,x_2,x_3)=(\eta_5\eta_9\eta_{10},\eta_2^2\eta_5,\eta_2^3,\eta_2\eta_5\eta_{9})
\end{equation*}
with $\eta_5+\eta_9+\eta_{10}=0$. Then we have $[\mathbf{x}]\in X_5$.

Type $\mathbf{D_5}$. The surface $X_6$ is defined by the equation
\begin{equation*}
x_3x_0^2+x_0x_2^2+x_1^2x_2=0.
\end{equation*}
The point $\mathbf{x}$ defined by
\begin{equation*}
\mathbf{x}=(x_0,x_1,x_2,x_3)=(\eta_4^3,\eta_4^2\eta_9,\eta_4^2\eta_8,\eta_8\eta_{10})
\end{equation*}
with $\eta_{10}+\eta_4\eta_8+\eta_{9}^2=0$ satisfies $[\mathbf{x}]\in X_6$.

Type $\mathbf{A_5+A_1}$. The defining equation for the surface $X_7$ is given by
\begin{equation*}
x_1^3+x_2x_3^2+x_0x_1x_2=0.
\end{equation*}
We write
\begin{equation*}
\mathbf{x}=(x_0,x_1,x_2,x_3)=(\eta_6\eta_{10},\eta_6\eta_{8},\eta_8^3,\eta_6\eta_{9}),
\end{equation*}
where $\eta_{6}+\eta_{9}^2+\eta_8\eta_{10}=0$. Then we obtain $[\mathbf{x}]\in X_7$.

Type $\mathbf{E_6}$. The equation
\begin{equation*}
x_1x_2^2+x_2x_0^2+x_3^3=0
\end{equation*}
defines the surface $X_8$. We set
\begin{equation*}
\mathbf{x}=(x_0,x_1,x_2,x_3)=(\eta_2^2\eta_{3}^2\tilde{\eta}_{6}^3,\tilde{\eta}_{6}^6\eta_{10},\eta_2^3\eta_3^4,\eta_2^2\eta_3^3\tilde{\eta}_{6}^2)
\end{equation*}
with $\eta_{10}+\eta_{2}+\eta_3=0$. It is easy to check that $[\mathbf{x}]$ lies on $X_8$.  We remark that $\tilde{\eta}_{6}$ corresponds to $\eta_6^{-1}$ in \cite{Der}.

So far, we establish parametrisations of rational points on each singular cubic surface $X_i$.

\section{The circle method}
In this section, we apply the circle method, combined with some ideas of Liu \cite{Liu} to investigate prime solutions to the equation (\ref{F1}). Note that the equation (\ref{F1}) is not linear and there is one term which is not diagonal.

Suppose that the polynomial  $F_1(t_1,t_2,t_3,t_4)\in\mathbb{Z}[t_1,t_2,t_3,t_4]$ is defined by
\begin{equation*}
F_1(t_1,t_2,t_3,t_4)=2\beta_1 t_1+\beta_2t_2^2+\beta_3t_3t_4.
\end{equation*}
Let $\gamma_1,\dots,\gamma_4$ be fixed positive real numbers satisfying
\begin{equation*}
F_1(\gamma_1,\gamma_2,\gamma_3,\gamma_4)=2\beta_1\gamma_1+\beta_2\gamma_2^2+\beta_3\gamma_3\gamma_4=0.
\end{equation*}
We write
\begin{equation*}
I_1=[\gamma_1(1-\delta) B^{\frac{2}{3}},\gamma_1(1+\delta)B^{\frac{2}{3}}]
\end{equation*}
and
\begin{equation*}
I_j=[\gamma_j\sqrt{1-\delta} B^{\frac{1}{3}},\gamma_j\sqrt{1+\delta}B^{\frac{1}{3}}]
\end{equation*}
for $2\leq j\leq 4$, where $B$ is a sufficiently large integer and $\delta$ is a fixed small positive constant. Furthermore, we set
\begin{equation*}
R_1(B)=\sum_{\stackrel{p_j\in I_j}{F_1(p_1,p_2,p_3,p_4)=0}}(\log p_1)(\log p_2)(\log p_3)(\log p_4),
\end{equation*}
which counts the weighted number of prime solutions to the equation $F(t_1,t_2,t_3,t_4)=0$ with $t_j\in I_j$ for $1\leq j\leq 4$. We apply the circle method to prove the following result.
\begin{lem}
\label{lem1}
For any $A>0$, we have
\begin{equation*}
R_1(B)=J(B)\mathfrak{S}+O(B(\log B)^{-A}),
\end{equation*}
where $J(B)$ is the number of integer solutions to the equation
\begin{equation*}
F(m_1,m_2,m_3,m_4)=0
\end{equation*}
with $m_j\in I_j$ for $1\leq j\leq 4$ and
\begin{equation*}
\mathfrak{S}= 2\prod_{p>2}\Big(1-\frac{1}{(p-1)^2}\Big).
\end{equation*}
Moreover, we have $J(B)\gg B$ and $\mathfrak{S}\gg1$.
\end{lem}
\begin{proof}
We set
\begin{equation*}
L=\log B,\qquad P=L^D, \qquad Q=B^{\frac{2}{3}}P^{-3},
\end{equation*}
where $D$ is a sufficiently large parameter to be chosen later. Furthermore, denote
\begin{equation*}
S_1(\alpha)=\sum_{p_1\in I_1}\log p_1 e(2\beta_1p_1\alpha ),\quad S_2(\alpha)=\sum_{p_2\in I_2}\log p_2 e(\beta_2p_2^2\alpha )
\end{equation*}
and
\begin{equation*}
S_3(\alpha)=\sum_{\stackrel{p_3\in I_3}{p_4\in I_4}}\log p_3\log p_4 e(\beta_3p_3p_4\alpha ).
\end{equation*}
Then the application of the circle method starts with the identity
\begin{equation*}
R_1(B)=\int_{\frac{1}{Q}}^{1+\frac{1}{Q}}S_1(\alpha)S_2(\alpha)S_3(\alpha)d \alpha.
\end{equation*}
By Dirichlet's lemma on rational approximation, we know that each $\alpha\in (\frac{1}{Q},1+\frac{1}{Q}]$ can be written in the form
\begin{equation*}
\alpha=\frac{a}{q}+\lambda,\quad |\lambda|<\frac{1}{qQ},
\end{equation*}
for some integers $a$, $q$ satisfying $1\leq a\leq q\leq Q$ and $(a,q)=1$. Define the sets of major and minor arcs as follows:
\begin{equation*}
\mathfrak{M}=\bigcup_{q\leq P}\bigcup_{\stackrel{1\leq a\leq q}{(a,q)=1}}\Big[\frac{a}{q}-\frac{1}{qQ},\frac{a}{q}+\frac{1}{qQ}\Big], \quad \mathfrak{m}=\big(\frac{1}{Q},1+\frac{1}{Q}\big]\setminus \mathfrak{M}.
\end{equation*}
Thus we get
\begin{equation*}
R_1(B)=\int_{\mathfrak{M}}S_1(\alpha)S_2(\alpha)S_3(\alpha)d \alpha+\int_{\mathfrak{m}}S_1(\alpha)S_2(\alpha)S_3(\alpha)d \alpha.
\end{equation*}
Now we estimate the integral over the major arcs. For any $\alpha\in\mathfrak{M}$, we know that there exist integers $a$ and $q$ satisfying
\begin{equation*}
\alpha=\frac{a}{q}+\lambda,\quad 1\leq q \leq P,\quad (a,q)=1 \quad \text{and}\quad |\lambda|<\frac{1}{qQ}.
\end{equation*}
Therefore, we get
\begin{equation*}
\begin{split}
S_1(\alpha)&=\sum_{p_1\in I_1}\log p_1 e\Big(\frac{2\beta_1ap_1}{q}\Big)e(2\beta_1\lambda p_1)\\
&=\frac{1}{\varphi(q)}\sum_{\chi\,  \text{mod}\,  q}\sum_{\stackrel{h_1=1}{(h_1,q)=1}}^{q}e\Big(\frac{2\beta_1ah_1}{q}\Big)\bar{\chi}(h_1)\sum_{p_1\in I_1}\log p_1\chi(p_1)e(2\beta_1\lambda p_1).
\end{split}
\end{equation*}
Let
\begin{equation*}
W_1(\chi,\lambda)=\sum_{m_1\in I_1}(\Lambda(m_1)\chi(m_1)-\delta_{\chi})e(2\beta_1\lambda m_1)
\end{equation*}
and
\begin{equation*}
\widehat{W}_1(\chi,\lambda)=\sum_{p_1\in I_1}\log p_1 \chi(p_1)e(2\beta_1\lambda p_1)-\sum_{m_1\in I_1}\delta_{\chi}e(2\beta_1\lambda m_1),
\end{equation*}
where $\delta_{\chi}=1$ or $0$ according as the character $\chi$ is principal or not. We obtain
\begin{equation}
\label{C1}
W_1(\chi,\lambda)-\widehat{W}_1(\chi,\lambda)=\sum_{k\geq2}\sum_{p_1^k\in I_1}\log p_1\chi(p_1^k)e(2\beta_1\lambda p_j^k)\ll B^{\frac{1}{3}}L.
\end{equation}
It follows that
\begin{equation}
\label{C2}
\begin{split}
&S_1(\alpha)-\frac{1}{\varphi(q)}\sum_{\stackrel{h_1=1}{(h_1,q)=1}}^{q}e\Big(\frac{2\beta_1ah_1}{q}\Big)\sum_{m_1\in I_1}e(2\beta_1\lambda m_1)\\
=&\frac{1}{\varphi(q)}\sum_{\chi\,\text{mod}\,q}\sum_{\stackrel{h_1=1}{(h_1,q)=1}}^q e\Big(\frac{2\beta_1ah_1}{q}\Big)\bar{\chi}(h_1)\widehat{W}_1(\chi,\lambda)\\
=&\frac{1}{\varphi(q)}\sum_{\chi\,\text{mod}\,q}\sum_{\stackrel{h_1=1}{(h_1,q)=1}}^q e\Big(\frac{2\beta_1ah_1}{q}\Big)\bar{\chi}(h_1)(\widehat{W}_1(\chi,\lambda)-W_1(\chi,\lambda))\\
&+\frac{1}{\varphi(q)}\sum_{\chi\,\text{mod}\,q}\sum_{\stackrel{h_1=1}{(h_1,q)=1}}^q e\Big(\frac{2\beta_1ah_1}{q}\Big)\bar{\chi}(h_1)W_1(\chi,\lambda).
\end{split}
\end{equation}
The main tool used here is the Siegel-Walfisz theorem. By using (6) in Perelli and Pintz \cite{P}, we see that
\begin{equation*}
\sum_{m_1\in I_1}\Lambda(m_1)\chi(m_1)=\sum_{m_1\in I_1}\delta_{\chi}+O(B^{\frac{2}{3}}L^{-2-5D}).
\end{equation*}
Therefore, integration by parts derives
\begin{equation*}
\begin{split}
W_1(\chi,\lambda)=&\int_{I_1}e(2\beta_1\lambda u)d\big( \sum_{m_1\leq u,m_1\in I_1}(\Lambda(m_1)\chi(m_1)-\delta_{\chi})\big)\\
\ll& \big|\sum_{m_1\in I_1}(\Lambda(m_1)\chi(m_1)-\delta_{\chi})\big|\\
&+\Big|\lambda\int_{I_1}e(2\beta_1\lambda u)\big( \sum_{m_1\leq u,m_1\in I_1}(\Lambda(m_1)\chi(m_1)-\delta_{\chi})\big) du\Big|\\
\ll& (1+|\lambda|B^{\frac{2}{3}})B^{\frac{2}{3}}L^{-2-5D}.
\end{split}
\end{equation*}
Thus we obtain
\begin{equation}
\label{C3}
W_1(\chi,\lambda)\ll (1+|\lambda|B^{\frac{2}{3}})B^{\frac{2}{3}}L^{-2-5D}\ll B^{\frac{2}{3}}L^{-2-2D}.
\end{equation}
Combining (\ref{C1}), (\ref{C2}) with (\ref{C3}), we get
\begin{equation}
\label{C4}
S_1(\alpha)=\frac{1}{\varphi(q)}\sum_{\stackrel{h_1=1}{(h_1,q)=1}}^{q}e\Big(\frac{2\beta_1ah_1}{q}\Big)\sum_{m_1\in I_1}e(2\beta_1\lambda m_1)+O(B^{\frac{2}{3}}L^{-2D}).
\end{equation}
Arguing similarly, we may get
\begin{equation}
\label{C5}
S_2(\alpha)=\frac{1}{\varphi(q)}\sum_{\stackrel{h_2=1}{(h_2,q)=1}}^qe\Big(\frac{\beta_2ah_2^2}{q}\Big)\sum_{m_2\in I_2}e(\beta_2\lambda m_2^2)+O(B^{\frac{1}{3}}L^{-2D}).
\end{equation}
Now we apply similar ideas as in Section 5 of \cite{Liu} to estimate $S_3(\alpha)$. For any fixed $p_3\in I_3$, we have
\begin{equation*}
\sum_{p_4\in I_4}\log p_4 e(\beta_3p_3p_4\alpha )=\frac{1}{\varphi(q)}\sum_{\stackrel{h_4=1}{(h_4,q)=1}}^{q}e\Big(\frac{\beta_3ap_3h_4}{q}\Big)\sum_{m_4\in I_4}e(\beta_3p_3m_4\lambda)+O(B^{\frac{1}{3}}L^{-2D}).
\end{equation*}
Then summing over $p_3\in I_3$, we get
\begin{equation*}
\begin{split}
S_3(\alpha)&=\frac{1}{\varphi(q)}\sum_{m_4\in I_4}\sum_{p_3\in I_3}\sum_{\stackrel{h_4=1}{(h_4,q)=1}}^{q}e\Big(\frac{\beta_3ap_3h_4}{q}\Big)\log p_3e(\beta_3p_3m_4\lambda)+O(B^{\frac{2}{3}}L^{-2D})\\
&=\frac{1}{\varphi^2(q)}\sum_{\stackrel{h_3=1}{(h_3,q)=1}}^{q}\sum_{\stackrel{h_4=1}{(h_4,q)=1}}^{q}e\Big(\frac{\beta_3ah_3h_4}{q}\Big)\sum_{{\stackrel{m_3\in I_3}{m_4\in I_4}}}e(\beta_3m_3m_4\lambda)+O(B^{\frac{2}{3}}L^{-2D})\\
&=\frac{\mu(q)}{\varphi(q)}\sum_{{\stackrel{m_3\in I_3}{m_4\in I_4}}}e(\beta_3m_3m_4\lambda)+O(B^{\frac{2}{3}}L^{-2D}).
\end{split}
\end{equation*}
Combining this with (\ref{C4}) and (\ref{C5}), we obtain
\begin{equation}
\label{34}
\begin{split}
&\int_{\mathfrak{M}}S_1(\alpha)S_2(\alpha)S_3(\alpha)d \alpha \\
=&\sum_{q\leq P}\frac{\mu(q)}{{\varphi}^3(q)}\sum_{\stackrel{a=1}{(a,q)=1}}^{q}\sum_{\stackrel{h_1=1}{(h_1,q)=1}}^{q}e\Big(\frac{2\beta_1ah_1}{q}\Big)\sum_{\stackrel{h_2=1}{(h_2,q)=1}}^{q}e\Big(\frac{\beta_2ah_2^2}{q}\Big)\\
&\times\int_{-\frac{1}{qQ}}^{\frac{1}{qQ}}\sum_{m_j\in I_j}e(\lambda F(m_1,m_2,m_3,m_4))d\lambda+O(BL^{-D}).
\end{split}
\end{equation}
Note that
\begin{equation*}
\begin{split}
&\int_{-\frac{1}{2}}^{\frac{1}{2}}\sum_{m_j\in I_j}e(\lambda F(m_1,m_2,m_3,m_4))d\lambda\\
&-\int_{-\frac{1}{qQ}}^{\frac{1}{qQ}}\sum_{m_j\in I_j}e(\lambda F(m_1,m_2,m_3,m_4))d\lambda\ll BL^{-D}.
\end{split}
\end{equation*}
Inserting this into (\ref{34}), we obtain
\begin{equation}
\label{majorarc}
\int_{\mathfrak{M}}S_1(\alpha)S_2(\alpha)S_3(\alpha)d \alpha=J(B)\mathfrak{S}(P)+O(BL^{-D}),
\end{equation}
where 
\begin{equation*}
\mathfrak{S}(P)=\sum\limits_{q\leq P}A(q)
\end{equation*}
and
\begin{equation}
\label{Aq} A(q)=\frac{\mu(q)}{{\varphi}^3(q)}\sum_{\stackrel{a=1}{(a,q)=1}}^{q}\sum_{\stackrel{h_1=1}{(h_1,q)=1}}^{q}e\Big(\frac{2\beta_1ah_1}{q}\Big)\sum_{\stackrel{h_2=1}{(h_2,q)=1}}^{q}e\Big(\frac{\beta_2ah_2^2}{q}\Big).
\end{equation}
Next we estimate the contribution of the integral over the minor arcs. For any $\alpha\in\mathfrak{m}$, there exist integers $a$ and $q$ such that
\begin{equation*}
P\leq q\leq Q,\quad (a,q)=1 \quad \text{and} \quad \Big| \alpha-\frac{a}{q}\Big|<\frac{1}{qQ}.
\end{equation*}
It follows from the non-trivial upper bound for the exponential sums over primes that for any $\alpha\in\mathfrak{m}$, we have
\begin{equation}
\label{minor}
S_2(\alpha)\ll B^{\frac{1}{3}}(\log B)^{-A},
\end{equation}
provided that $D$ is chosen to be sufficiently large. Such a result can be found in several references, for example see Theorem 2 in Ghosh \cite{Ghosh}. Moreover, it follows from the trivial estimate that we have
\begin{equation}
\label{Mean1}
\int_{\frac{1}{Q}}^{1+\frac{1}{Q}}|S_1(\alpha)|^2 d\alpha \ll B^{\frac{2}{3}}.
\end{equation}
Furthermore, we see that the term $\int_{\frac{1}{Q}}^{1+\frac{1}{Q}}|S_3(\alpha)|^2 d\alpha $ counts the weighted number of the solutions the equation
\begin{equation*}
  p_3p_4=p_3'p_4'
\end{equation*}
with $p_3,p_3'\in I_3$ and $p_4,p_4'\in I_4$. Then we have
\begin{equation}
\label{Mean2}
\int_{\frac{1}{Q}}^{1+\frac{1}{Q}}|S_3(\alpha)|^2 d\alpha \ll B^{\frac{2}{3}}.
\end{equation}
By Cauchy's inequality and (\ref{minor})--(\ref{Mean2}), we obtain
\begin{equation}
\label{minorarc}
\begin{split}
&\int_{\mathfrak{m}}S_1(\alpha)S_2(\alpha)S_3(\alpha)d \alpha\\
&\ll \sup_{\alpha\in \mathfrak{m}}S_2(\alpha)\Big(\int_{\frac{1}{Q}}^{1+\frac{1}{Q}}|S_1(\alpha)|^2\Big)^{\frac{1}{2}}\Big(\int_{\frac{1}{Q}}^{1+\frac{1}{Q}}|S_3(\alpha)|^2\Big)^{\frac{1}{2}}\\
&\ll BL^{-A}.
\end{split}
\end{equation}
Combining (\ref{majorarc}) and (\ref{minorarc}), we get
\begin{equation*}
R_1(B)=J(B)\mathfrak{S}(P)+O(BL^{-A}),
\end{equation*}
provided that the parameter $D$ is chosen to be sufficiently large in terms of $A$. For the singular series, we establish the upper bound for $A(q)$ defined as in (\ref{Aq}). Note that 
\begin{equation*}
  \sum_{\stackrel{h_1=1}{(h_1,q)=1}}^{q}e\Big(\frac{2\beta_1ah_1}{q}\Big)=  \left\{\begin{array}{ll}
\mu(q),\ & \text{if} \ 2\nmid q,\\
\mu(\frac{q}{2}),\ & \text{if} \ 2| q,
    \end{array}\right. \\
 \end{equation*}
under the assumption that $\mu(q)\neq0$ and $(a,q)=1$. Moreover, we have
\begin{equation*}
\sum_{\stackrel{a=1}{(a,q)=1}}^{q}\sum_{\stackrel{h_2=1}{(h_2,q)=1}}^{q}e\Big(\frac{\beta_2ah_2^2}{q}\Big)=\mu(q)\varphi(q).
\end{equation*}
Then it follows from the trivial estimate that
\begin{equation*}
|A(q)|\leq \frac{1}{{\varphi^2(q)}}.
\end{equation*}
Write
\begin{equation*}
\mathfrak{S}=\sum_{q=1}^{\infty}A(q).
\end{equation*}
We remark that $\mathfrak{S}$ can be viewed as the product of the local densities, i.e.
\begin{equation*}
\begin{split}
 \mathfrak{S}= & \prod_{p}\Big(1-\frac{1}{(p-1)^3}\sum_{a=1}^{p-1}\sum_{h_1=1}^{p-1}e\Big(\frac{2\beta_1ah_1}{p}\Big)\sum_{h_2=1}^{p-1}e\Big(\frac{\beta_2ah_2^2}{p}\Big)\Big) \\
     = &2\prod_{p>2}\Big(1-\frac{1}{(p-1)^2}\Big).
\end{split}
\end{equation*}
Hence we get $\mathfrak{S}\gg 1$. Then we obtain
\begin{equation*}
\mathfrak{S}(P)-\mathfrak{S}\ll\sum_{q>P}q^{-2+\epsilon}\ll L^{-\frac{D}{2}}.
\end{equation*}
Since $J(B)\ll B$, we get
\begin{equation*}
R_1(B)=J(B)\mathfrak{S}+O(B(\log B)^{-A}),
\end{equation*}
provided that $D$ is chosen to be sufficiently large in terms of $A$. Now we establish the lower bound for $J(B)$. For $j=2,3,4$, we restrict $m_j$ to be odd integers satisfying $m_j\in I_j$. Recall that
\begin{equation*}
2\beta_1\gamma_1+\beta_2\gamma_2^2+\beta_3\gamma_3\gamma_4=0,
\end{equation*}
thus we have $\beta_2m_2^2+\beta_3m_3m_4$ is even and
\begin{equation*}
\beta_2m_2^2+\beta_3m_3m_4\in[-2\beta_1\gamma_1(1-\delta) B^{\frac{2}{3}},-2\beta_1\gamma_1(1+\delta)B^{\frac{2}{3}}].
\end{equation*}
Since
\begin{equation*}
I_1=[\gamma_1 (1-\delta)B^{\frac{2}{3}},\gamma_1(1+\delta)B^{\frac{2}{3}}],
\end{equation*}
we deduce that
\begin{equation*}
J(B)\gg B.
\end{equation*}
Hence we prove the lemma.
\end{proof}
\section{Linear equations in primes}
In this section, we treat the linear equations (\ref{F2})--(\ref{F4}). We first record certain auxiliary results on linear equations in primes due to Green and Tao \cite{GT}.

Let $\mathbb{Z}_q=\mathbb{Z}/q\mathbb{Z}$ be the cyclic group of order $q$. Moreover, we let $\Lambda_{\mathbb{Z}_q}:\mathbb{Z}\rightarrow\mathbb{R}^{+}$ be the local von Mangoldt function, which is the $q$-periodic function defined by
\begin{equation*}
\Lambda_{\mathbb{Z}_q}(b)=
\begin{cases}
  \frac{q}{\phi(q)}, & \mbox{if } (b,q)=1, \\
  0, & \mbox{otherwise}.
\end{cases}
\end{equation*}
For any integer $N\geq1$, we write $[N]$ for the discrete interval $\{1,\dots,N\}$. Assume that $A$ is a finite nonempty set and $f:A\rightarrow \mathbb{C}$ is a function. We define the average of $f$ on $A$ to be
\begin{equation*}
  \mathbb{E}_{x\in A}f(x)=\frac{1}{|A|}\sum_{x\in A}f(x),
\end{equation*}
where $|A|$ is the cardinality of $A$. The following lemma is essentially a restatement of Theorem 1.8 of Green and Tao \cite{GT}, which counts the weighted number of prime solutions in a given range to a system of linear equations.
\begin{lem}
\label{LEP}
Assume the inverse Gowers-norm conjecture GI($s$) and the M\"{o}bius and nilsequences conjecture MN($s$) hold. Let $A=(a_{ij})$ be an $s\times t$ matrix of integers, where $s\leq t$. Assume the nondegeneracy conditions that $A$ has full rank $s$, and that the only element of the row-space of $A$ over $\mathbb{Q}$ with two or fewer nonzero entries is the zero vector. Let $N>1$, let $\mathbf{b}=(b_1,\dots,b_s)\in \mathbb{Z}^s$ be a vector in $A\mathbb{Z}^t=\{A\mathbf{x}:\mathbf{x}\in\mathbb{Z}^t\}$, and suppose that the coefficients $|a_{ij}|$ and the quantities $|b_i/N|$ are uniformly bounded by some constant $L$. Let $K\subset[-N,N]^t$ be convex. Then we have
\begin{equation}
\label{estimate}
  \sum_{\stackrel{\mathbf{x}\in K\cap \mathbb{Z}^t}{A\mathbf{x}=\mathbf{b}}}\prod_{i\in[t]}\Lambda(x_i)=\alpha_{\infty}\prod_p \alpha_p+o_{t,L,s}(N^{t-s}),
\end{equation}
where the local densities $\alpha_p$ are given by
\begin{equation*}
  \alpha_p:=\lim_{M\rightarrow\infty}\mathbb{E}_{x\in[-M,M]^t,A\mathbf{x}=\mathbf{b}}\prod_{i\in[t]}\Lambda_{\mathbb{Z}_p}(x_i)
\end{equation*}
and the global factor $\alpha_\infty$ is given by
\begin{equation*}
  \alpha_{\infty}:=\#\{\mathbf{x}\in\mathbb{Z}^t:\mathbf{x}\in K,A\mathbf{x}=\mathbf{b},x_i\geq0\}.
\end{equation*}
\end{lem}
For the case $s=1$, the conjectures GI($s$) and MN($s$) hold due to the work of Hardy-Littlewood and Vinogradov. While for $s=2$, the conjectures GI($s$) and MN($s$) are settled in \cite{GT2} and \cite{GT3}, respectively. Thus Lemma \ref{LEP} holds unconditionally for $s=1,2$. We apply Lemma \ref{LEP} to investigate prime solutions to particular linear equations.

Now we investigate the prime solutions to the equation (\ref{F2}). Let the polynomial  $F_2(t_1,t_2,t_3)\in\mathbb{Z}[t_1,t_2,t_3]$ be defined by
\begin{equation*}
F_2(t_1,t_2,t_3)=2\beta_1 t_1+\beta_2t_2+\beta_3t_3.
\end{equation*}
Suppose that $\gamma_1,\dots,\gamma_3$ are fixed positive real numbers, satisfying
\begin{equation}
\label{gamma}
F_2(\gamma_1,\gamma_2,\gamma_3)=2\beta_1\gamma_1+\beta_2\gamma_2+\beta_3\gamma_3=0.
\end{equation}
We set
\begin{equation*}
I_j=[\gamma_j(1-\delta) B^{\alpha},\gamma_j(1+\delta)B^{\alpha}]
\end{equation*}
for $1\leq j\leq 3$, where $B$ is a sufficiently large parameter, $\alpha\in(0,1)$ and $\delta$ is a fixed small positive constant. Furthermore, we set
\begin{equation*}
R_2(B)=\sum_{\stackrel{m_j\in I_j}{F_2(m_1,m_2,m_3)=0}}\Lambda (m_1)\Lambda (m_2)\Lambda (m_3),
\end{equation*}
which counts the weighted number of prime solutions to the equation (\ref{F2}) with $p_j\in I_j$ for $1\leq j\leq 3$. Then we choose $\mathbf{b}=\mathbf{0}$, $K=I_1\times I_2\times I_3$ and $A$ to be a $1\times3$ matrix $(2\beta_1\, \beta_2\, \beta_3)$ in Lemma \ref{LEP}. It follows that the estimate (\ref{estimate}) holds, i.e.
\begin{equation*}
  R_2(B)=\alpha_{\infty}\prod_p \alpha_p+o(B^{2\alpha}).
\end{equation*}
Since there is no local obstruction for the equation (\ref{F2}), the local density $\alpha_p$ is always positive for any prime $p$. Furthermore, we deduce from (\ref{gamma}) that the global factor $\alpha_\infty$ satisfies $\alpha_\infty\gg B^{2\alpha}$. Thus the inequality $R_2(B)\gg B^{2\alpha}$ holds.

Next we consider the prime solutions to the equation (\ref{F3}). Define $F_3(t_1,t_2,t_3,t_4)\in\mathbb{Z}[t_1,t_2,t_3,t_4]$ to be
\begin{equation*}
F_3(t_1,t_2,t_3,t_4)=\beta_1 t_1+\beta_2t_2+\beta_3t_3+\beta_4t_4.
\end{equation*}
Let $\gamma_1,\dots,\gamma_4$ be fixed positive real numbers with
\begin{equation*}
\beta_1\gamma_1+\beta_2\gamma_2+\beta_3\gamma_3+\beta_4\gamma_4=0.
\end{equation*}
For $1\leq j\leq 4$, we write
\begin{equation*}
I_j=[\gamma_j(1-\delta) B^{\frac{1}{3}},\gamma_j(1+\delta)B^{\frac{1}{3}}],
\end{equation*}
where $B$ is a sufficiently large parameter and $\delta$ is a fixed small positive constant. On writing $A=(\beta_1\, \beta_2\, \beta_3\,\beta_4)$, we argue similarly and deduce that
\begin{equation*}
R_3(B)=\sum_{\stackrel{m_j\in I_j}{F_3(m_1,m_2,m_3,m_4)=0}}\Lambda (m_1)\Lambda (m_2)\Lambda (m_3)\Lambda (m_4)\gg B.
\end{equation*}

Finally we consider prime solutions to the system of linear equations (\ref{F4}). Write
\begin{equation}
\label{F_45}
\begin{split}
   F_4(t_1,\dots,t_5)&=  \beta_1t_1+2\beta_4t_4-\beta_5t_5,\\
    F_5(t_1,\dots,t_5) &= \beta_2t_2+\beta_3t_3+2\beta_4t_4.
\end{split}
\end{equation}
Suppose that $\gamma_1,\dots,\gamma_5$ are fixed positive real numbers satisfying $F_4(\gamma_1,\dots,\gamma_5)=F_5(\gamma_1,\dots,\gamma_5)=0$. Set
\begin{equation}
\label{interval}
I_j=[\gamma_j(1-\delta) B^{\frac{1}{3}},\gamma_j(1+\delta)B^{\frac{1}{3}}]
\end{equation}
for $1\leq j\leq 5$, where $B$ is a sufficiently large parameter and $\delta$ is a fixed small positive constant. Let $A$ be a $2\times5$ matrix
\begin{gather*}
\begin{pmatrix}
  \beta_1 & 0 & 0 & 2\beta_4 & -\beta_5 \\
  0 & \beta_2 & \beta_3 & 2\beta_4 & 0
\end{pmatrix}.
\end{gather*}
Then it follows from Lemma \ref{LEP} that
\begin{equation}
\label{weighted}
R_4(B)=\sum_{\stackrel{m_j\in I_j}{F_4(m_1,\dots,m_5)=F_5(m_1,\dots,m_5)=0}}\Lambda (m_1)\dots\Lambda (m_5)\gg B.
\end{equation}
\section{Completion of the proof}
In this section, we establish the Zariski density of almost prime points on cubic surfaces and complete the proof of Theorem \ref{theorem}. Since the treatment for each $X_i$ is similar, we restrict our attention to the surface $X_2$ and only give the sketch of the proof for other cubic surfaces.

For the surface $X_2$, we recall the parametrisations in (\ref{para2}). We fix $\eta_{1}=\beta_1p_1$, $\eta_8=\beta_2p_2$, $\eta_9=\beta_3p_3$ and $\eta_{10}=2\beta_4p_4$ with $p_1,\dots,p_4$ distinct. Then the point given in (\ref{para2}) becomes
\begin{equation*}
\mathbf{x}=(x_0,x_1,x_2,x_3)=(2\beta_2\beta_3\beta_4p_2p_3p_4,\beta_2p_1^2p_2,\beta_3p_1^2p_3,p_1^2(\beta_1p_1+2\beta_4p_4))
\end{equation*}
with $\beta_2p_2+\beta_3p_3+2\beta_4p_4=0$. In order to obtain the upper bound for $r(X_2)$, we introduce the condition $\beta_1p_1+2\beta_4p_4-\beta_5p_5=0$. 
Given any real point $[\boldsymbol{\xi}_2]\in U_2$, we have
\begin{equation*}
  \xi_0(\xi_1+\xi_2+\xi_3)^2+\xi_1\xi_2(\xi_1+\xi_2)=0,
\end{equation*}
where $\xi_1\xi_2\neq0$ and $\xi_1+\xi_2+\xi_3\neq0$. Now we fix
\begin{equation*}
\begin{split}
   \gamma_1=|\sqrt[3]{\xi_1+\xi_2+\xi_3}&|,\,\gamma_2=\Big|\frac{\xi_1}{\sqrt[3]{(\xi_1+\xi_2+\xi_3)^2}}\Big|,\, \gamma_3=\Big|\frac{\xi_2}{\sqrt[3]{(\xi_1+\xi_2+\xi_3)^2}}\Big|,\\
   \gamma_4=&\Big|\frac{\xi_0\sqrt[3]{(\xi_1+\xi_2+\xi_3)^4}}{2\xi_1\xi_2}\Big|,\,\gamma_5=\Big|\frac{\xi_3}{\sqrt[3]{(\xi_1+\xi_2+\xi_3)^2}}\Big|
\end{split}
\end{equation*}
and
\begin{equation*}
\begin{split}
    \beta_1=\text{sgn}&(\sqrt[3]{\xi_1+\xi_2+\xi_3}),\, \beta_2=\text{sgn}(\xi_1),\, \beta_3=\text{sgn}(\xi_2),\\
\beta_4=&\text{sgn}\Big(\frac{\xi_0}{\xi_1\xi_2}\Big),\, \beta_5=\text{sgn}(\xi_3).
\end{split}
\end{equation*}
It is easy to check that the positive real numbers $\gamma_1,\dots,\gamma_5$ satisfy $F_4(\gamma_1,\dots,\gamma_5)=F_5(\gamma_1,\dots,\gamma_5)=0$, where $F_4$ and $F_5$ are defined as in (\ref{F_45}). Moreover, for $1\leq j\leq5$, we define the intervals $I_j$ as in (\ref{interval}). Then for $p_j\in I_j$ and $0\leq i\leq 3$, we obtain that
\begin{equation*}
  \Big|\frac{x_i}{B}-\xi_i\Big|< \delta f_i(\boldsymbol{\xi}_2),
\end{equation*}
where each $f_i(\boldsymbol{\xi}_2)$ is some nonnegative function in $\boldsymbol{\xi}_2$. Then we may choose $\delta$ to be sufficiently small, such that
\begin{equation*}
  \delta f_i(\boldsymbol{\xi}_2)<\varepsilon.
\end{equation*}
It follows from (\ref{weighted}) that for sufficiently large $B$, there exists a suitable positive constant $c$ such that there are at least $cB(\log B)^{-5}$ prime solutions to the equation (\ref{F4}) with $p_j\in I_j$. Among these solutions, there are at most $O(B^{\frac{2}{3}})$ solutions, for which $\mathbf{x}\notin\mathbb{Z}^4_{\text{prim}}$. Thus we deduce that the estimate (\ref{eq2}) holds.

For other cubic surfaces, we now present the parametrisations of the almost prime points.

Type $\mathbf{D}_4$, first isomorphy class. We set $\eta_1=\beta_1p_1$, $\eta_8=\beta_2p_2$, $\eta_9=\beta_3p_3$ and $\eta_{10}=\beta_4p_4$. Then we write
\begin{equation*}
\mathbf{x}=(x_0,x_1,x_2,x_3)=(\beta_2\beta_3\beta_{4}p_2p_3p_4,\beta_2p_1^2p_2,\beta_3p_1^2p_3,\beta_4p_1^2p_4),
\end{equation*}
where $\beta_2p_2+\beta_3p_3+\beta_4p_4-\beta_1p_1=0$. Then we have $[\mathbf{x}]\in X_1$ and $x_0x_1x_2x_3=P_{12}$.

Type $\mathbf{A_3+2A_1}$. We fix $\eta_1=\beta_1p_1$, $\eta_4=\beta_2p_2$, $\eta_8=2\beta_3p_3$, $\eta_{10}=\beta_4p_4$. Thus the point defined by
\begin{equation*}
\mathbf{x}=(x_0,x_1,x_2,x_3)=(\beta_1p_1p_4^2,2\beta_3p_2^2p_3,2\beta_1\beta_2\beta_3p_1p_2p_3,\beta_1\beta_2\beta_{4}p_1p_2p_4)
\end{equation*}
with $\beta_1p_1+\beta_2p_2+2\beta_3p_3=0$ satisfies $[\mathbf{x}]\in X_3$ and $x_0x_1x_2x_3=P_{14}$.

Type $\mathbf{2A_2+A_1}$. Let $\eta_5=\beta_1p_1$, $\eta_6=\beta_2p_2$, $\eta_{10}=2\beta_3p_3$ and $\eta_2=\beta_4p_4$. On writing
\begin{equation*}
\mathbf{x}=(x_0,x_1,x_2,x_3)=(\beta_1p_1p_2^2,\beta_4p_1^2p_4,2\beta_3p_3p_4^2,\beta_1\beta_2\beta_{4}p_1p_2p_4),
\end{equation*}
where $\beta_1p_1+\beta_2p_2+2\beta_3p_3=0$, we get $[\mathbf{x}]\in X_4$ and $x_0x_1x_2x_3=P_{13}$.

Type $\mathbf{A_4+A_1}$. We choose $\eta_5=\beta_1p_1$, $\eta_9=\beta_2p_2$, $\eta_{10}=2\beta_3p_3$ and $\eta_2=\beta_4p_4$. It follows that the point
\begin{equation*}
\mathbf{x}=(x_0,x_1,x_2,x_3)=(2\beta_1\beta_2\beta_{3}p_1p_2p_3,\beta_1p_1p_4^2,\beta_4p_4^3,\beta_1\beta_2\beta_{4}p_1p_2p_4)
\end{equation*}
with $\beta_1p_1+\beta_2p_2+2\beta_3p_3=0$ satisfies $[\mathbf{x}]\in X_5$ and $x_0x_1x_2x_3=P_{13}$.

Type $\mathbf{D_5}$. The point $\mathbf{x}$ defined by
\begin{equation*}
\mathbf{x}=(x_0,x_1,x_2,x_3)=(\beta_1p_1^3,\beta_3p_1^2p_3,\beta_2p_1^2p_2,2\beta_2\beta_{4}p_2p_4)
\end{equation*}
with $2\beta_4p_4+\beta_1\beta_2p_1p_2+p_3^2=0$ satisfies $[\mathbf{x}]\in X_6$ and $x_0x_1x_2x_3=P_{12}$.

Type $\mathbf{A_5+A_1}$. We take $\eta_6=2\beta_1p_1$, $\eta_8=\beta_2p_2$, $\eta_9=\beta_3p_3$ and $\eta_10=\beta_4p_4$. Write
\begin{equation*}
\mathbf{x}=(x_0,x_1,x_2,x_3)=(2\beta_1\beta_4p_1p_4,2\beta_1\beta_2p_1p_2,\beta_2p_2^3,2\beta_1\beta_3p_1p_3),
\end{equation*}
where $2\beta_1p_1+p_3^2+\beta_2\beta_4p_2p_4=0$. Then we obtain $[\mathbf{x}]\in X_7$ and $x_0x_1x_2x_3=P_{12}$.

Type $\mathbf{E_6}$. Fix $\eta_2=\beta_1p_1$, $\eta_3=\beta_2p_2$, $\eta_{10}=2\beta_3p_3$ and $\tilde{\eta}_{6}=\beta_4p_4$. Then the point defined by
\begin{equation*}
\mathbf{x}=(x_0,x_1,x_2,x_3)=(\beta_4p_1^2p_2^2p_4^3,2\beta_3p_3p_4^6,\beta_1p_1^3p_2^4,\beta_2p_1^2p_2^3p_4^2)
\end{equation*}
with $2\beta_3p_3+\beta_1p_1+\beta_2p_2=0$. We see that $[\mathbf{x}]$ lies on $X_8$ and $x_0x_1x_2x_3=P_{29}$.

Consequently, we apply the result on the number of prime solutions to the equation (\ref{F3}) and (\ref{F2}) to investigate almost prime points on $X_1$ and $X_3,X_4,X_5,X_8$, respectively. Moreover, we consider almost prime points on $X_6$ and $X_7$ by using Lemma \ref{lem1}. Arguing similarly as in the treatment for $X_2$, we may establish the Zariski density for almost prime points on each singular cubic surface. Then the proof of Theorem \ref{theorem} is concluded.

\end{document}